\documentclass[reqno,12pt]{amsart}
\usepackage{amsmath,amsfonts,amsthm,amssymb,indentfirst}
\usepackage[all,poly]{xy} 
\usepackage{color}
\usepackage[pagebackref,colorlinks]{hyperref}
\usepackage{kantlipsum}
\usepackage[all,poly]{xy}
\usepackage{mathrsfs}
\usepackage{enumerate}
\theoremstyle{plain}
\newtheorem{lemma}{Lemma}[section]

\newtheorem{proposition}[lemma]{Proposition}

\theoremstyle{definition}

\newtheorem{example}[lemma]{Example}

\newcommand{\Zset}{\mathbb Z}

\newcommand{\M}{\operatorname{\mathbb M}}
\newcommand{\gr}{\operatorname{gr}}
\newcommand{\ol}{\overline}

\newcommand{\End}{\operatorname{End}}
\newcommand{\END}{\operatorname{END}}
\newcommand{\Hom}{\operatorname{Hom}}
\newcommand{\HOM}{\operatorname{HOM}}

\newcommand{\so}{\mathbf{s}}
\newcommand{\ra}{\mathbf{r}}

\title[Cancellation properties of graded and nonunital rings]{Cancellation properties of graded and nonunital rings.\\ \smallskip Graded clean and graded exchange Leavitt path algebras}  

\author{Lia Va\v s}
\address{Department of Mathematics, Physics and Statistics, University of the Sciences, Philadelphia, PA 19104, USA}
\email{l.vas@usciences.edu}



\begin{document}
\begin{abstract} 
Various authors have been generalizing some unital ring properties to nonunital rings. We consider properties related to cancellation of modules (being unit-regular, having stable range one, being directly finite, exchange, or clean) and their ``local'' versions.  We explore their relationships and extend the defined concepts to graded rings. With graded clean and graded exchange rings suitably defined, we study how these properties behave under the formation of graded matrix rings. We exhibit properties of a graph $E$ which are equivalent to the unital Leavitt path algebra $L_K(E)$ being graded clean. We also exhibit some graph properties which are necessary and some which are sufficient for $L_K(E)$ to be graded exchange.
\end{abstract}

\maketitle

\section{Introduction}

The paper is motivated by two questions {\em ``If $P$ is a property of unital rings, how does one define a generalized version of $P$ suitable for nonunital rings?''} and  {\em ``If $P$ is a property of rings, how does one define the graded version $P_{\gr}$ of $P$ in a meaningful way?''} which we address for several different properties $P$. We say that a ring property is a {\em cancellation property} if it can be directly related to one of the cancellation properties of modules (internal cancellation, module-theoretic exchange, module-theoretic direct finiteness, and substitution) or if it is ``sandwiched'' between two properties directly relatable to module cancellation. For example, the properties that a ring is unit-regular, exchange, directly finite, or that it has stable range one can be directly related to internal cancellation, module-theoretic exchange, module-theoretic direct finiteness, and substitution respectively (see \cite{Lam_cancellation_properties} or \cite{Lia_cancellation} for more details). We consider cleanness to also be a cancellation property because clean rings bridge the classes of unit-regular and exchange rings.  

The cancellation properties mentioned above critically depend on the existence of the ring identity. Some of these properties have been adapted to nonunital rings: see \cite{Mary_Patricio} for unit-regularity, \cite{Vaserstein2} for stable range one, \cite{Ara_exchange} for exchange, and \cite{Nicholson_Zhou} for cleanness. In section \ref{section_nonunital}, we review the relevant definitions and adapt direct finiteness to nonunital rings. If Reg stands for von Neumann regular, and if UR stands for unit-regular, sr=1 for stable range one, DF for directly finite, Cln for clean, and Exch for exchange generalizations to possibly nonunital rings, the diagram below shows relations between these concepts. 
\[\label{D1}\tag{D1}
\begin{array}{lcll}
\text{UR} \Leftarrow & \text{Reg + sr=1} & \Rightarrow \text{sr=1} & \Rightarrow \text{ DF }\\
& \Downarrow & &\\
& \text{Cln} & \Rightarrow \text{Exch} &\\
\end{array}\]
We also consider the ``local'' versions of cancellation properties. If $P$ is a ring property, a ring $R$ has {\em local} $P$ if every finite subset of $R$ is contained in a corner of $R$ which has $P$. This definition is especially convenient for locally unital rings since finite subsets of such rings are contained in corners. We consider relations between $P$ and local $P$ when $P$ is a cancellation property. 

In section \ref{section_graded_rings}, we turn to rings which are graded. In \cite{Lia_cancellation}, the graded version of unit-regularity, stable range one, and direct finiteness were considered in the unital case. In this paper, we focus on the remaining two cancellation properties: being an exchange ring and being a clean ring. 
We introduce the graded versions of being clean in the unital case and being exchange in the general case and study properties of these generalized concepts, in particular how they behave under the formation of graded matrix rings or algebras (Propositions \ref{graded_clean_matrices} and \ref{matrices_graded_exchange}).  

In section \ref{section_LPAs}, we consider Leavitt path algebras. It turns out to be useful if one can pair a ring property $P$ with a graph property $P'$ so that the Leavitt path algebra $L_K(E)$ of a graph $E$ has $P$ if and only if $E$ has $P'.$ Such pairings enable one to create rings with various predetermined properties by choosing suitable graphs (for example prime and not primitive rings, simple and not purely infinite simple rings, and so on). A number of ring-theoretic properties have been paired up in such a way: being regular, simple, purely infinite simple, hereditary, semisimple, Artinian, Noetherian, directly finite, Baer, to name some of them. In particular, such characterizations are known for properties of being locally unit-regular, locally directly finite, and exchange. Proposition
\ref{local_P_implies_P} implies that the graph properties which characterize local unit-regularity and local direct finiteness also characterize the generalized versions of unit-regularity and direct finiteness. 

Leavitt path algebras are naturally $\Zset$-graded. Thus, pairing a {\em graded} ring property of $L_K(E)$ with a graph property of $E$ is also of interest. This has previously been done for graded local direct finiteness for any $E$ (\cite{Roozbeh_Ranga_Ashish}), and for graded unit-regularity and graded stable range one in the case when $E$ is finite (\cite{Lia_cancellation}). We relax this last assumption to the requirement that $E$ has finitely many vertices (no restriction on the number of edges) in Proposition \ref{graded_UR_characterization}. We present a necessary condition for a Leavitt path algebra to be  graded exchange (Proposition \ref{graded_exchange_necessary}) and a sufficient condition for the same property (Proposition \ref{graded_exchange_sufficient}) and we question if any of them is both necessary and sufficient. We also characterize which unital Leavitt path algebras are graded clean (Proposition \ref{LPA_graded_clean_characterization}). 
This latter result does not answer the long standing question {\em ``Which Leavitt path algebras are clean?''}, but it settles it in the graded case. 

The diagram below summarizes the relationship between the graded versions of the cancellation properties of Leavitt path algebras. 
If $P$ stands for any of UR, sr=1, DF, Cln, or Exch, then gr in the subscript indicates the graded version of these properties, and $0$ in the subscript denotes the condition that the $0$-component of the graded endomorphism ring of $L_K(E)$ has $P.$
The number of vertices of the graph is assumed to be finite for the vertical arrows except for those  related to Exch$_{\gr}$ and Reg$_{\gr}$ for which there are no restrictions on $E$.  

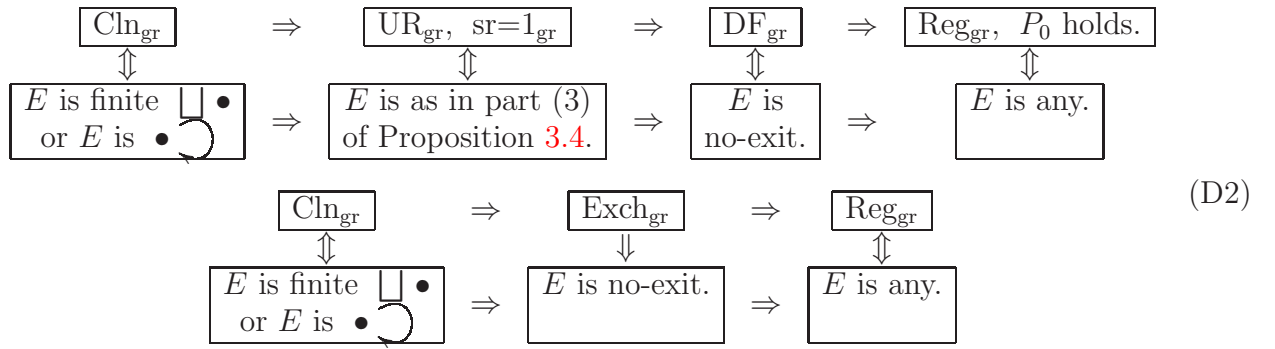
\begin{center}
\begin{tabular}{ccccccc}
\begin{tabular}{|c|} \hline 
Cln$_{\gr}$\\ \hline
\end{tabular}
& $\Rightarrow$ & 
\begin{tabular}{|c|} \hline 
UR$_{\gr},\;$ sr=1$_{\gr}$  \\ \hline
\end{tabular}
& $\Rightarrow$ & 
\begin{tabular}{|c|} \hline
DF$_{\gr}$ \\ \hline
\end{tabular} 
& $\Rightarrow$& 
\begin{tabular}{|c|} \hline 
Reg$_{\gr},\;$ $P_0$ holds. \\ \hline
\end{tabular}\\
$\Updownarrow$&&$\Updownarrow$&&$\Updownarrow$&&$\Updownarrow$\\
\begin{tabular}{|c|} \hline 
$E$ is finite $\;\bigsqcup\,\bullet$
\\or $E$ is $\xymatrix{\bullet\ar@(ru,rd) }\;\;\;\;$\\ \hline
\end{tabular}
& $\Rightarrow$ & 
\begin{tabular}{|c|} \hline 
$E$ is as in part (3)\\ of Proposition \ref{graded_UR_characterization}.  
\\ \hline
\end{tabular}
& $\Rightarrow$ & 
\begin{tabular}{|c|} \hline 
$E$ is\\ no-exit. 
\\ \hline
\end{tabular}
& $\Rightarrow$& 
\begin{tabular}{|c|} \hline 
$E$ is any. \\ \\ \hline
\end{tabular}\\
\end{tabular}
\end{center}
\[\label{D2}\tag{D2}\begin{array}{c}
\end{array}\]

\begin{center}\vskip-.5cm
\begin{tabular}{ccccc}
\begin{tabular}{|c|} \hline 
Cln$_{\gr}$\\ \hline
\end{tabular}
& $\Rightarrow$ & 
\begin{tabular}{|c|} \hline 
Exch$_{\gr}$   \\ \hline
\end{tabular}
& $\Rightarrow$& 
\begin{tabular}{|c|} \hline 
Reg$_{\gr}$  \\ \hline
\end{tabular}\\
$\Updownarrow$&&$\Downarrow$&&$\Updownarrow$\\
\begin{tabular}{|c|} \hline 
$E$ is finite $\;\bigsqcup\,\bullet$
\\or $E$ is $\xymatrix{\bullet\ar@(ru,rd) }\;\;\;\;$\\ \hline
\end{tabular}
& $\Rightarrow$ & 
\begin{tabular}{|c|} \hline 
$E$ is no-exit.\\   
\\ \hline
\end{tabular}
& $\Rightarrow$ & 
\begin{tabular}{|c|} \hline 
$E$ is any. \\ \\ \hline
\end{tabular}\\
\end{tabular}
\end{center}

Using the graph properties in the diagram, it is straightforward to produce examples showing that each horizontal implication is strict. This, in its own right, provides interesting examples illustrating differences between properties of rings and their graded generalizations. For example, Reg$_{\gr}\nRightarrow$ Exch$_{\gr},$ contrasts Reg $\Rightarrow$ Exch, and UR$_{\gr} \nRightarrow$ Cln$_{\gr},$ contrasts UR $\Rightarrow$ Cln.

\section{Cancellation properties and their nonunital generalizations}\label{section_nonunital}

Rings are assumed to be associative. If a ring has the identify, we call it a {\em unital ring}. We use the term a {\em general ring} to denote a possible lack of the identity. If $R$ is a unital ring, $U(R)$ denotes the set of elements which are invertible in $R$. 

\subsection{Unit-regular, stable range one, and directly finite rings}
A general ring is (von Neumann) {\em regular} if for every $x\in R,$ $x\in xRx.$ If $R$ is a unital ring,  $R$ is {\em unit-regular} if every $x\in R$ is unit-regular meaning that there is an invertible $u\in R$ such that $x=xux$; $R$ has {\em stable range one} if for every $x, y\in R$ such that $xR+yR=R$, there is $z\in R$ such that $(x+yz)R=R$ (equivalently, $x+yz$ is invertible, see \cite[Theorem 2.6]{Vaserstein2}); $R$ is {\em directly finite} if for every $x,y\in R,$ $xy=1$ implies $yx=1.$ If Reg, UR, sr=1, and DF denote these properties for short, the following relations hold and all implications are strict. For more details, see \cite{Lam_cancellation_properties} or \cite{Lia_cancellation}.
\begin{center}
Reg $\;\;\Leftarrow\;\;$ UR $\;\;\Rightarrow\;\;$ sr=1 $\;\;\Rightarrow\;\;$ DF 
\end{center}
All of the above properties are closed when passing to a corner, i.e. a ring of the form $eRe$ where $e\in R$ is an idempotent of a ring $R$ (\cite{Lam_cancellation_properties}, \cite{Vaserstein2} and \cite{Lia_cancellation} contain more details).  

\subsection{Exchange modules and rings}\label{subsection_exchange}
Recall that a right $R$-module $A$ has the {\em (finite) exchange property} if whenever (a copy of) $A$ is a direct summand of a right module $B=\bigoplus_{i\in I} B_i$ where $I$ is a (finite) index set, then there are submodules $A_i\leq B_i$ such that $\bigoplus_{i\in I} A_i$ is a complement of $A$ in $B.$ If $I$ is infinite, we denote this property by Exch($A$) and, if $I$ is finite, by FinExch($A$). 

An element $x$ of a unital ring $R$ is an {\em exchange element} of $R$ if there is an idempotent $e\in xR$ such that $1-e\in (1-x)R$ and $R$ is an {\em exchange ring} if every element is exchange. We use Exch$^r$ to shorten this last condition. 
The superscript $r$ indicates the presence of right modules. Exch$^l$ can be defined analogously. By \cite[Theorem 2]{Warfield_exchange} and \cite[Theorem 2.1]{Nicholson_clean}, the following conditions are equivalent for a right $R$-module $A$ (which we denote by writing $A_R$ for $A$).
\begin{enumerate}
\item[1.] FinExch($A_R$) holds. \hskip3.35cm 2. Exch$^r$ holds for $\End_R(A).$
\item[3.] Exch$^l$ holds for $\End_R(A).$ \hskip2.4cm 4. Exch($\End_R(A)_{\End_R(A)}$) holds.
\end{enumerate}
In the case when $A$ is $R_R$ and $R$ is unital, all of the above conditions are equivalent so one can suppress writing $l$ or $r$ in the superscript of Exch. 
By \cite[Lemma 3.10]{Crawley_Jonsson} and \cite[Proposition 1.10]{Nicholson_clean}, Exch is closed under the formation of matrix rings and corners of unital rings.   
By \cite[Proposition 1.1]{Nicholson_clean}, Exch is equivalent to the condition that for any $x\in R,$ there is an idempotent $e\in R$ such that $e-x\in (x-x^2)R.$ Since Exch is left-right symmetric, this condition is also left-right symmetric and we use Lift to denote it.

\subsection{Clean rings}\label{subsection_clean_rings}
An element $x$ of a unital ring $R$ is a {\em clean element} of $R$ if $x=u+e$ for some unit $u$ and some idempotent $e.$ In this case, the relation $x=u+e$ is a {\em clean decomposition} of $x.$ A ring is {\em clean} if every element is clean and we use Cln to denote this requirement. The property Cln is closed under the formation of matrix rings (see \cite[Corollary 1]{Han_Nicholson}), it is not closed under the formation of corners (see \cite{Ster_corners}), and it is not Morita invariant (see \cite{Ster}). 

The implications below hold by \cite{Camillo_Khurana} and \cite{Nicholson_clean} and we review the argument for the second one. 
\begin{center}
UR $\;\;\Rightarrow\;\;$ Cln $\;\;\Rightarrow\;\;$ Exch 
\end{center}
If $x=e+u$ is a clean decomposition of an element $x$ of a clean ring $R,$  then $f=u^{-1}(1-e)u$ is an idempotent such that 
$(x-x^2)u^{-1}=(e+u)u^{-1}-(eu^{-1}+ueu^{-1}+e+u)=f-e-u=f-x$
and so the property Lift holds for $R.$ Both implications are strict by examples from \cite{Camillo_Yu}.

\subsection{The standard unitization and operations \texorpdfstring{$\ast$}{TEXT} and \texorpdfstring{$\circ$}{TEXT}} 
If $R$ is a general ring, a unital ring $S$ such that $R$ embeds in $S$ as a double-sided ideal of $S$ is an {\em unitization} of $R.$ The {\em standard unitization} $R^u$ of $R$ is the ring $R\oplus \Zset$ with the addition given coordinate-wise and the multiplication given by
\[(x, k)(y, l)=(xy+lx+ky, kl)\]
so that $(0,1)$ is the identity of $R^u.$ If $R$ is a $K$-algebra for a field $K,$ $K$ can be used instead of $\Zset.$ 

The map $x\mapsto -x$ is an isomorphism of the monoids $(R, \ast)$ and $(R, \circ)$ where   
\[x\ast y=x+y+xy\hskip.5cm\mbox{ and }\hskip.5cm x\circ y=x+y-xy.\]
The ordered pairs in $U(R^u)$ are exactly the elements of the form $\pm(x,1)$ for $x\in U(R,\ast)$ (equivalently, the elements of the form $\pm(x,-1)$ for $x\in U(R,\circ)).$ 

\subsection{Nonunital unit-regular, stable range one, clean, and exchange rings} We continue to use UR, sr=1, Cln, Exch and DF for the following generalizations of the corresponding cancellation properties of unital rings. 

Operation $\ast$ was used in \cite{Mary_Patricio} to generalize UR: a general ring $R$ is {\em unit-regular} if for every $x\in R,$ there is $u\in U(\ast)$ such that $x=xux+x^2.$  The implication UR $\Rightarrow$ Reg continues to hold (see the proof of \cite[Theorem 4.1]{Mary_Patricio}).  By \cite[Example 4.3]{Mary_Patricio}, UR is not closed under the formation of corners. 

In \cite{Vaserstein2}, a general ring $R$ is said to have {\em stable range one} if for all $x\in R, \ol y\in R^u,$ such that $(x,1)R^u+\ol y R^u=R^u,$ there is $\ol z\in R^u$ such that $((x,1)+\ol y\, \ol z)R^u=R^u.$ By \cite[Theorem 3.6]{Vaserstein2}, this is equivalent with: for any $x,y\in R$ such that $(x, 1)R^u+(y, 0)R^u=R^u,$ there is $z\in R$ such that $(x+yz, 1)R^u=R^u.$ It is direct to check that this last condition is equivalent to the following: for any $x,y\in R$ such that $0\in x\ast R+yR,$ there is $z\in R$ such that $0\in (x+yz)\ast R.$ 

Next, we show that Reg + sr=1 $\Rightarrow$ UR. The converse does not hold by \cite[Example 4.3]{Mary_Patricio}.  
 
\begin{proposition}
For general rings, {\em Reg + sr=1} $\Rightarrow$ {\em UR}. 
\label{UR_vs_sr_one} 
\end{proposition}
\begin{proof}
Let $R$ be a regular general ring with stable range one. By \cite[Theorem 2]{Funayama}, $R$ has a regular unitization $S.$ If we identify $R$ with its image in $S$, $R$ is contained in the directed union of the corners $eSe$ for $e=e^2\in R$ by \cite[Lemma 1.1]{Menal_Moncasi}. We also recall \cite[Lemma 1.4]{Menal_Moncasi} which states that a general ring $I$ with a regular unitization $T$ has sr=1 if and only if $eTe$ is unit-regular for every $e=e^2\in I$. By this result, any $x\in R$ is contained in a unit-regular corner $eSe$ for some $e=e^2\in R.$ If $u\in U(eSe)$ is such that $x=xux,$ then $1-e+u\in U(S)$ is such that $x=x(1-e+u)x.$ Thus, $x$ is unit-regular in $S.$ By Proposition \cite[Corollary 4.5]{Mary_Patricio}, $R$ is unit-regular.
\end{proof}

In \cite{Nicholson_Zhou}, a general ring $R$ is said to be {\em clean} if for any $x\in R,$ there is an idempotent $e\in R$ and $u\in U(\ast)$ such that $x=e+u.$ By \cite[Proposition 7]{Nicholson_Zhou}, this requirement is equivalent to suitable conditions expressed in terms of unitizations of $R.$ 
The property Cln does not transfer to corners even in the unital case by \cite{Ster_corners}. By \cite[Theorem 2]{H_Chen}, Reg+sr=1 $\Rightarrow$ Cln. We also note that cleanness is generalized to nonunital rings in \cite{Ilic-Gregorijevic_UJ} using an approach different than in \cite{Nicholson_Zhou}. 

In \cite{Ara_exchange}, a general ring $R$ is said to be an {\em exchange} ring if for any $x\in R,$ there is an idempotent $e\in xR$ such that  $e\in x\circ R.$ By \cite[Theorem 1.2]{Ara_exchange}, this requirement is left-right symmetric. The implication Cln $\;\Rightarrow\;$ Exch continues to hold by  \cite[Theorem 2]{Nicholson_Zhou}. Exch is closed under the formation of corners (by \cite[Proposition 1.3]{Ara_exchange}) and direct limits (direct to show). 

\subsection{Nonunital direct finiteness} Just as UR, the condition DF is given in terms of multiplication only, without any reference to addition. Thus, one can define a monoid $(M, \cdot, 1)$ to be directly finite if for every $x,y\in M,$ $xy=1$ implies $yx=1$ and 
define that  a general ring $R$ is {\em directly finite} if any of the conditions from Proposition \ref{directly_finite} holds.

\begin{proposition} If $R$ is a general ring, the following conditions are equivalent. 
\begin{enumerate}[\upshape(1)]
\item $(R,\ast)$ is directly finite. \hskip1cm {\em (2)} $(R, \circ)$ is directly finite.\hskip1cm {\em (3)} $R^u$ is directly finite.

\item[{\em (4)}]  For any embedding $\phi:R\to S$ such that $S$ is an unitization of $R,$ and for any $x,y\in R,$ $(\phi(x)+1)(\phi(y)+1)=1$ implies $(\phi(y)+1)(\phi(x)+1)=1.$
\end{enumerate}
If $R$ is unital, then the above conditions are equivalent with $R$ being directly finite. 
\label{directly_finite}
\end{proposition}
\begin{proof}
The equivalence of (1) and (2) holds since the monoids  $(R,\ast)$ and $(R, \circ)$ are isomorphic. 

To show (1)  $\Rightarrow$ (3), assume that $(x,k)(y,l)=(0,1)$ holds in $R^u.$ Then $k=l=1$ or $k=l=-1.$ In the first case, $x\ast y=x+y+xy=0$ which implies that $y\ast x=0$ by (1), and so $(y,1)(x,1)=(0,1)$. In the second case, $-(x\circ y)=-x-y+xy=0$ and so $x\circ y=x+y-xy=0$ which implies that $y\circ x=0$ by (2) and so $(y, -1)(x,-1)=(0,1)$. Thus, (3) holds. The converse (3) $\Rightarrow$ (1) is direct since $x\ast y=0$ readily implies $(x, 1)(y, 1)=(0,1).$  

To show (1) $\Rightarrow$ (4), assume that (1) holds and let $\phi$ and $S$ be as in (4). If $(\phi(x)+1)(\phi(y)+1)=1$ holds in $S,$ then $\phi(x+y+xy)=0$ so $x\ast y=0$ which implies that   $y\ast x=0$ by (1). Taking $\phi$ of both sides and adding 1 produces $(\phi(y)+1)(\phi(x)+1)=1.$ The converse (4) $\Rightarrow$ (1) is similar. 

If $R$ is unital, $xy=1$ iff $(x-1)\ast (y-1)=0$ and $x\ast y=0$ iff $(x+1)(y+1)=1$ for any $x,y\in R.$ This implies the last sentence of the proposition. 
\end{proof}

The property DF retains some favorable properties of the unital DF as we show next.  
 
\begin{proposition}
The property {\em DF} is closed under the formation of corners and $\,$ {\em sr=1} $\Rightarrow$ {\em DF}. 
\label{DF_and_corners}
\end{proposition}
\begin{proof}
To show the claim on corners, let $R$ be a general ring, $e$ an idempotent of $R$, and $x,y\in eRe$ be such that $xy=e.$ Then $(x-e, 1)(y-e, 1)=(xy-x-y+e+x-e+y-e, 1)=(xy-e, 1)=(0,1)$ and so $(y-e, 1)(x-e, 1)=(0,1)$ which implies that $yx=e.$

To show sr=1 $\Rightarrow$ DF, assume that a general ring $R$ has sr=1 and let $x,y\in R$ be such that $x\ast y=0.$ Then $(x,1)(y,1)=(0,1)$ in $R^u$ and so $\bar e=(0,1)-(y,1)(x,1)=(-x-y-yx,0)$ is an idempotent in $R^u$ such that $(0,1)=\bar e+(y,1)(x,1)\in \bar eR^u+(y,1)R^u.$ As sr=1 holds, there is $\bar z\in R^u$ such that $((y,1)+\bar e \bar z)R^u=R^u.$ Since $(x,1)\bar e=(x,1)-(x,1)(y,1)(x,1)=(x,1)-(x,1)=(0,0),$ $(x,1)((y,1)+\bar e\bar z)=(x,1)(y,1)=(0,1).$ So, $(y,1)+\bar e\bar z$ is left and right invertible and, thus, its left inverse $(x,1)$ is invertible. So, the right inverse $(y,1)$ of $(x,1)$ is the double-sided inverse of $(x,1).$ The relation $(y,1)(x,1)=(0,1)$ implies that $y\ast x=0.$
\end{proof}

By Propositions \ref{UR_vs_sr_one} and \ref{DF_and_corners} and by results we reviewed above, all implications of diagram (\ref{D1}) hold. By \cite[Example 4.3]{Mary_Patricio}, the relation $\Leftarrow$ is strict. Examples showing that three horizontal $\Rightarrow$  are strict are known to exist among unital rings. Example of a clean ring with a non-clean corner from \cite{Ster_corners} provides an example of a clean ring which does not have Reg+sr=1. 

\subsection{Local cancellation properties}\label{subsection_local}

Recall that a general ring $R$ is said to be a {\em locally unital} ring (and to have {\em local units}) if for every finite set $F,$ there is an idempotent $e$ such that $F\subseteq eRe.$    

If $P$ a property of (unital or general) rings, we say that a general ring $R$ has {\em local $P$} if every finite set of elements of $R$ is contained in a corner of $R$ which has $P$.   
This ``local'' approach has been used for generalizing unital UR in  \cite{Gene_Ranga_regular} and unital DF in  \cite{Lia_traces}. In \cite{Roozbeh_Lia_Baer}, the properties of being Baer and Rickart were generalized using an approach equivalent to this one (since being Baer and being Rickart transfer to corners).  

If a general ring $R$ has local $P,$ then it has local units. If $P$ transfers to corners and $R$ is a locally unital ring which has $P,$ then $R$ has local $P$ (since any finite set is contained in a corner and every corner has $P$). We show the converse of $(P$ $\Rightarrow$ local $P)$ for all five cancellation properties. 

\begin{proposition}
If $P$ is any of {\em UR, sr=1, DF, Cln,} or {\em Exch}, then (local $P$ $\Rightarrow$ $P)$ holds.  
\label{local_P_implies_P}
\end{proposition}
\begin{proof}
For UR, let $x\in R$ and let $e\in R$ be idempotent such that $x\in eRe$ and $eRe$ is unit-regular. Let $u\in U(eRe)$ be such that $x=xux.$ Then, $u-e\in U(\ast)$ and $x(u-e)x+x^2=xux-x^2+x^2=x.$ Hence, $x$ is a unit-regular element of $R.$ 

For sr=1, let $x,y\in R$ be such that $0\in x\ast R+yR$ and let $e=e^2$ be such that $x,y\in eRe$ and that $eRe$ has stable range one. If $0=x\ast u+yv=x+u+xu+yv$ for some $u,v\in R,$ then $x+ue+xue+yve=0.$ As $x, y, xue, yve\in eRe,$ $ue$ is in $eRe$ and $yve=yeve.$ Thus, $0=x\ast ue +yve\in x\ast eRe+yeRe.$ By the assumption that sr=1 holds for $eRe,$ there is $z\in eRe\subseteq R$ such that $0\in (x+yz)\ast eRe\subseteq (x+yz)\ast R.$

For DF, let $x,y\in R$ and let $e=e^2\in R$ be such that $x,y\in eRe$ and $eRe$ is DF. If $x\ast y=0,$ then $(x+e)(y+e)=e$ and so  $(y+e)(x+e)=e$ which implies $y\ast x=0$.  

For Cln, let $x\in R$ and $e=e^2\in R$ be such that $x\in eRe$ and $eRe$ is clean. As $e+x\in eRe,$ let $e+x=f+u$ for some $f=f^2\in eRe$ and $u\in U(eRe).$ Then $x=f+(u-e)$ and $u-e\in U(\ast).$ 

For Exch, if $x\in R$ and $e\in R$ is such that $x\in eRe$ and $eRe$ is exchange, then there is $f=f^2\in eRe$ such that $f\in xeRe\subseteq xR$ and $f\in x\circ eRe\subseteq x\circ R.$ Thus, $x$ is exchange in $R.$ 
\end{proof}

\section{Cancellation properties of graded rings}\label{section_graded_rings}

Unless stated otherwise, $\Gamma$ denotes an arbitrary group and $\varepsilon$ denotes its identity element. 

\subsection{Graded rings prerequisites}\label{subsection_graded_rings_prerequisites}
A general ring $R$ is {\em graded} by a group $\Gamma$ if $R=\bigoplus_{\gamma\in\Gamma} R_\gamma$ for additive subgroups $R_\gamma$ and $R_\gamma R_\delta\subseteq R_{\gamma\delta}$ for all $\gamma,\delta\in\Gamma.$ The elements of the set $H=\bigcup_{\gamma\in\Gamma} R_\gamma$ are said to be {\em homogeneous}. The grading is {\em trivial} if $R_\gamma=0$ for every $\varepsilon\neq\gamma\in \Gamma.$ We adopt the standard definitions of graded ring homomorphisms, graded left and right $R$-modules, graded module homomorphisms, graded algebras, graded left and right ideals, graded left and right free and projective modules as defined in \cite{NvO_book} and \cite{Roozbeh_book}. A $\Gamma$-graded unital ring $R$ is a {\em graded division ring} if every $0\neq x\in H$ is invertible.  In this case, $R$ is a {\em graded field} if $R$ is also commutative.

If $M$ is a graded right $R$-module and $\gamma\in\Gamma,$ the $\gamma$-\emph{shifted or $\gamma$-suspended} graded right $R$-module $(\gamma)M$ is defined as the module $M$ with the $\Gamma$-grading given by $(\gamma)M_\delta = M_{\gamma\delta}$ for any $\delta\in \Gamma.$ 
If $M$ and $N$ are graded right $R$-modules and $\gamma\in\Gamma$, then $\Hom_R(M,N)_\gamma$ denotes the following 
\[\Hom_R(M,N)_\gamma=\{f\in \Hom_R(M, N)\mid f(M_\delta)\subseteq N_{\gamma\delta}\mbox{ for any }\delta\in\Gamma\},\]
then the subgroups $ \Hom_R(M,N)_\gamma$ of $\Hom_R(M,N)$ intersect trivially and $\HOM_R(M,N)$ denotes their direct sum $\bigoplus_{\gamma\in\Gamma} \Hom_R(M,N)_\gamma.$ The notation $\END_R(M)$ is used in the case if $M=N.$ If $M$ is finitely generated (which is the case we often consider), then $\Hom_R(M,N)=\HOM_R(M,N)$ for any $N$ (both \cite{NvO_book} and \cite{Roozbeh_book} contain details) and $\End_R(M)=\END_R(M)$ is a $\Gamma$-graded unital ring.

In \cite{Roozbeh_book}, for a $\Gamma$-graded unital ring $R$ and $\gamma_1,\dots,\gamma_n\in \Gamma$, $\M_n(R)(\gamma_1,\dots,\gamma_n)$ denotes the ring of matrices $\M_n(R)$ with the $\Gamma$-grading given by  
$(r_{ij})\in\M_n(R)(\gamma_1,\dots,\gamma_n)_\delta$ if $r_{ij}\in R_{\gamma_i^{-1}\delta\gamma_j}$ for $i,j=1,\ldots, n.$ 
In \cite{NvO_book}, $\M_n(R)(\gamma_1,\dots,\gamma_n)$ is defined so that it is $\M_n(R)(\gamma_1^{-1},\dots,\gamma_n^{-1})$ using the definition from \cite{Roozbeh_book} (more details on the relations between two definitions can be found in \cite[Section 1]{Lia_realization}). We opt to use the definition from \cite{Roozbeh_book}. With this definition, if $F$ is the graded free right module $(\gamma_1^{-1})R\oplus \dots \oplus (\gamma_n^{-1})R$ (and any finitely generated graded free right $R$-module is of that form, see  \cite{NvO_book} or \cite{Roozbeh_book}), then $\Hom_R(F,F)\cong_{\gr} \;\M_n(R)(\gamma_1,\dots,\gamma_n)$ as graded unital rings.      

We also recall \cite[Theorem 1.3.3]{Roozbeh_book}  stating the following lemma for $\Gamma$ abelian only, but the proof generalizes to arbitrary $\Gamma.$ \cite[Remark 2.10.6]{NvO_book} also states the first two parts for arbitrary $\Gamma.$

\begin{lemma} \cite[Theorem 1.3.3]{Roozbeh_book}. Let $R$ be a $\Gamma$-graded unital ring and $\gamma_1,\ldots,\gamma_n\in \Gamma.$ 
\begin{enumerate}[\upshape(1)]
\item If $\pi$ a permutation of the set $\{1,\ldots, n\},$ then 
\begin{center}
$\M_n (R)(\gamma_1, \gamma_2,\ldots, \gamma_n)\;\cong_{\gr}\;\M_n (R)(\gamma_{\pi(1)}, \gamma_{\pi(2)} \ldots, \gamma_{\pi(n)}).$
\end{center}
\item If $\delta$ in the center of $\Gamma,$ $\;\M_n (R)(\gamma_1, \gamma_2, \ldots, \gamma_n)\;=\;\M_n (R)(\gamma_1\delta, \gamma_2\delta,\ldots, \gamma_n\delta).$
\item If $\delta\in\Gamma$ is such that there is an invertible element $u_\delta$ in $R_\delta,$ then  
\begin{center}
$\M_n (R)(\gamma_1, \gamma_2, \ldots, \gamma_n)\;\cong_{\gr}\;\M_n (R)(\gamma_1\delta, \gamma_2\ldots, \gamma_n).$
\end{center}
\end{enumerate}
\label{lemma_on_shifts} 
\end{lemma}

We recall some examples we often consider in the rest of the paper. Let $K$ be a field and consider it to be trivially graded by $\Zset.$ By Lemma \ref{lemma_on_shifts}, every graded matrix algebra over $K$ is graded isomorphic to $\M_n(K)(0, l_1,\ldots,l_{n-1})$ where $l_i$ are integers such that $0\leq l_i\leq l_{i+1}$ for $i=1,\ldots, n-2.$

If $m$ is a positive integer, let $K[x^m, x^{-m}]$ denote the ring of Laurent polynomials over $x^m.$ We consider this ring $\Zset$-graded by $K[x^m, x^{-m}]_{km}=Kx^{km}$ and $K[x^m, x^{-m}]_{km+l}=0$ for any integer $k$ and $l\in \{1,\ldots, m-1\}.$ By Lemma \ref{lemma_on_shifts}, 
every graded matrix ring over $K[x^m, x^{-m}]$ is graded isomorphic to $\M_n(K[x^m, x^{-m}])(0, l_1,\ldots,l_{n-1})$ where $l_i\in \{0, 1, \ldots, m-1\}$ and $l_i\leq l_{i+1}$ for $i=1,\ldots, n-2.$

\subsection{Cancellation properties of graded rings}\label{subsection_properties_of_graded_rings} 
If $P$ is a (general or unital) ring property, the term {\em graded property $P$} has been used for the property  
$P_{\gr}$ obtained by replacing every $\forall x$ and $\exists x$ appearing in $P$ by the restricted versions $\forall x\in H$ and $\exists x\in H$ where $H$ is the set of homogeneous elements. For example, since the property $(\forall x)(\exists y)(xyx=x)$ defines a regular ring, a graded ring $R$ is {\em graded regular} if $(\forall x\in H)(\exists y\in H)(xyx=x)$ and we denote this condition by Reg$_{\gr}.$

Similarly, if $R$ is a graded unital ring, then $R$ is {\em graded unit-regular} if for every $x\in H,$ there is $u\in U(R)\cap H$ such that $xux=x;$ 
$R$ has {\em graded stable range one} if the condition $(\gamma^{-1})xR+(\delta^{-1})yR=R$ for some $x\in R_\gamma, y\in R_\delta,$ implies that there is $z\in H$ such that $(\gamma^{-1})(x+yz)R=R;$ $R$ is {\em graded directly finite} if $xy=1$ implies $yx=1$ for all $x,y\in H.$ 
By \cite{Lia_cancellation}, $\,$ UR$_{\gr}$ $\;\Rightarrow\;$ sr=1$_{\gr}$ $\;\Rightarrow\;$  DF$_{\gr}.$

The graded versions of Cln and Exch can be defined similarly. In particular, a  graded unital ring $R$ is {\em graded clean} if every homogeneous $x\in R$ is a sum of a homogeneous unit and a homogeneous idempotent and we call such sum a {\em graded clean decomposition} of $x$. If $R$ is graded clean, we say that Cln$_{\gr}$ holds for $R.$  Such definition was also considered in \cite[Definition 3.7]{Ilic-Georgijevic_Sahinkaya}.

We say that $R$ is {\em graded right exchange} and that Exch$_{\gr}^r$ holds for $R$ if for every homogeneous $x\in R,$ there is a homogeneous idempotent $e\in R$ such that $e\in xR$ and $1-e\in (1-x)R.$ The condition  Exch$_{\gr}^l$ is defined analogously. The ring $R$ is graded exchange if it is both graded left and graded right exchange and we say that Exch$_{\gr}$ holds in this case. 

If the grading is trivial and $P$ is any cancellation property, the condition $P_{\gr}$ reduces to $P.$ We examine the graded versions of Cln and Exch in more details next. 

\subsection{Graded cleanness} The condition Cln$_{\gr}$ is very restrictive as the following lemma shows.
The direction $\Rightarrow$ of the lemma was also noted in \cite[Remark 3.8]{Ilic-Georgijevic_Sahinkaya}.

\begin{lemma} Let $R$ be a $\Gamma$-graded unital ring. Then $R$ is graded clean if and only if $R_\varepsilon$ is clean and each nonzero element of $R_\gamma$ is invertible for every $\gamma\neq \varepsilon.$ 
\label{graded_clean_lemma}
\end{lemma}
\begin{proof}
If Cln$_{\gr}$ holds for $R$ then $R_\varepsilon$ is clean. If $0\neq x\in R_\gamma$ for some $\gamma\neq \varepsilon$ and $x=u+e$ for some $e=e^2\in R_\varepsilon$ and $u\in R_\delta\cap U(R),$ then $\gamma=\delta$ and $e=0.$ Thus, $x=u$ is invertible.  

Conversely, let $x\in R_\gamma.$ If $\gamma=\varepsilon,$ then $x$ is in $R_\varepsilon$ so a clean decomposition of $x$ in $R_\varepsilon$ is a graded clean decomposition of $x$ in $R.$  
If $\gamma\neq\varepsilon$ and $x=0,$ $x=1+(-1)$ is a graded clean decomposition. If $\gamma\neq \varepsilon$ and $x\neq 0,$ then $x\in U(R)$ and so $x=0+x$ is a graded clean decomposition. 
\end{proof}

By this lemma, a graded division ring is graded clean. The converse does not hold (take a clean ring which is not a division ring and grade it trivially). 

We use Lemma \ref{graded_clean_lemma} to characterize when a graded matrix ring is graded clean next. 

\begin{proposition} Let $n$ be a positive integer, $\gamma_1,\gamma_2,\ldots,\gamma_n\in\Gamma,$ and $R$ be a $\Gamma$-graded unital ring.  
\begin{enumerate}[\upshape(1)]
\item If $R$ is trivially graded, then $\M_n(R)(\gamma_1, \gamma_2\ldots,\gamma_n)$ is graded clean if and only if $\gamma_1=\gamma_2=\ldots=\gamma_n$ and $\M_n(R)$ is clean.
\item If $R$ is not trivially graded, then $\M_n(R)(\gamma_1, \gamma_2\ldots,\gamma_n)$ is graded clean if and only if $n=1$ and $R$ is graded clean.  
\end{enumerate}
\label{graded_clean_matrices} 
\end{proposition}
\begin{proof}
Let $S=\M_n(R)(\gamma_1, \gamma_2\ldots,\gamma_n).$ 
The standard matrix unit $e_{ij}$ is in $S_{\gamma_i\gamma_j^{-1}}$ so it is homogeneous. 

(1) If $S$ is graded clean and $\gamma_i\neq \gamma_j,$ then $e_{ij}$ is not invertible and it is in $S_{\gamma_i\gamma_j^{-1}}\neq S_\varepsilon.$ By Lemma \ref{graded_clean_lemma} this cannot happen and so $\gamma_i=\gamma_j.$ Since $R$ is trivially graded, $S_\varepsilon=\M_n(R_\varepsilon)=\M_n(R).$ The assumption that $S$ is graded clean implies that $S_\varepsilon=\M_n(R)$ is clean by Lemma \ref{graded_clean_lemma}. To show the converse, assume that $\gamma_1=\gamma_2=\ldots=\gamma_n$ and that $\M_n(R)=S_\varepsilon$ is clean. Then $S_\gamma=0$ for all $\gamma\neq \varepsilon.$ Thus, $S$ is graded clean by Lemma \ref{graded_clean_lemma}. 

(2) As $R$ is not trivially graded, there is $\delta\neq \varepsilon$ so that $R_{\gamma_1^{-1}\delta\gamma_1}\neq 0.$ Let $0\neq x\in R_{\gamma_1^{-1}\delta\gamma_1}.$ If $n>1,$ then  $xe_{11}\in S_\delta$ is noninvertible. If $S$ is graded clean, this cannot happen, so $n=1.$ Since $S=\M_1(R)(\gamma_1)\cong_{gr} (\gamma_1^{-1})R(\gamma_1)$ is graded clean, so is $R.$ Conversely, if $n=1,$ then $S\cong_{\gr}(\gamma_1^{-1})R(\gamma_1).$ The graded cleanness of $R$ implies the graded cleanness of $(\gamma_1^{-1})R(\gamma_1)$. 
\end{proof}

Proposition \ref{graded_clean_matrices} shows that Cln$_{\gr}$ is not closed under the formation of graded matrix rings. The property Cln$_{\gr}$ is also not closed under the formation of direct sums. Indeed, $K[x, x^{-1}]$ is graded clean but $(x,0)$ is a homogeneous element of $R=K[x, x^{-1}]\oplus K[y, y^{-1}]$ which is not invertible, so $R$ is not graded clean by Lemma \ref{graded_clean_lemma}. 

The next example shows that the conditions UR$_{\gr}$ and Cln$_{\gr}$ are independent properties. 
 
\begin{example}\label{example_not_morita}
Let $K$ be a field trivially graded by $\Zset,$ let $K[x^2, x^{-2}]$ be naturally $\Zset$-graded (see section \ref{subsection_graded_rings_prerequisites}), and let $R=\M_2(K[x^2, x^{-2}])(0, 1).$ By \cite[Proposition 5.1]{Lia_cancellation}, $R$ is graded unit-regular. By Proposition \ref{graded_clean_matrices}, $R$ is not graded clean. Thus, UR$_{\gr}\nRightarrow$ Cln$_{\gr}.$

If a unital ring which is clean but not unit-regular (for example, the ring from \cite[Section 3]{Ster}) is graded trivially by any group, we obtain an example showing that Cln$_{\gr}\nRightarrow$ UR$_{\gr}.$
\end{example}
 
The pairs
(Cln, Cln$_{\gr}$) and (Exch, Exch$_{\gr}$) are pairs of mutually independent properties. 

\begin{example}
The $\Zset$-graded ring $\M_2(K)(0,1)$ is clean (because $\M_2(K)$ is clean) but not graded clean by Proposition \ref{graded_clean_matrices}.  

The ring $K[x, x^{-1}]$ is a graded field so it is graded clean and graded exchange. However, $K[x, x^{-1}]$ is neither clean nor exchange. One can see this last fact by representing this algebra as a Leavitt path algebra of the graph $\xymatrix{\bullet\ar@(ru,rd) }\;\;\;\;\,$ and using  \cite[Theorem 4.5]{Gonzalo_Pardo_Molina_exchange}. 

The Leavitt algebra $L_K(1, 2)$ is a universal example of a $K$-algebra $R$ such that $R^2\cong R,$ and, also, it is the Leavitt path algebra of the graph $\;\;\;\;\xymatrix{\ar@(lu,ld)\bullet\ar@(ru,rd) }\;\;\;\;$. This algebra is exchange (by \cite[Theorem 4.5]{Gonzalo_Pardo_Molina_exchange}) and neither graded left nor graded right exchange by Proposition \ref{graded_exchange_necessary} (Example \ref{example_gr_reg_not_gr_exch} has more details).
\end{example}  

\subsection{Unital graded exchange}
\label{subsection_unital_graded_exchange} 
If the grading is trivial, Exch$_{\gr}^r$ and Exch$_{\gr}^l$ are equivalent. Recall that this holds by the proof of 
\begin{center}
Exch$^r$ holds for $\End_R(R_R)\;\Rightarrow\;$ FinExch$(R_R)$ holds $\;\Rightarrow\;$ Exch$^l$ holds for $\End_R(R_R).$ 
\end{center}
This proof does not transfer to the graded case. Also, the condition ``FinExch$(R_R)$ holds'' is equivalent with $\End_R(R_R)_\varepsilon$ being exchange which is strictly weaker than Exch$_{\gr}^r$ or Exch$_{\gr}^l$ holding on $\End_R(R_R)$  (Example \ref{example_gr_reg_not_gr_exch} shows this). However, if $R$ and its opposite ring are graded isomorphic, then the conditions Exch$_{\gr}^r$ and Exch$_{\gr}^l$ are equivalent. This is the case when $R$ is a graded involutive ring (i.e. if $R$ has an involution $*$ such that $(R_\gamma)^*\subseteq R_{\gamma^{-1}}$ for any $\gamma\in\Gamma$). The rings we focus on (graded matrix algebras over graded fields and Leavitt path algebras) are graded involutive.  

If $R$ is a graded general ring, let Lift$^r_{\gr}$ denote the condition that for every homogeneous $x\in R,$ there is a homogeneous idempotent $e \in R$ such that $e-x\in (x-x^2)R.$ If $R$ is unital, Exch$^r_{\gr}$ and Lift$^r_{\gr}$ are equivalent. Indeed, following the proof of Exch $\Leftrightarrow$ Lift from \cite[Proposition 1.1]{Nicholson_clean}, one can take $e$ from Exch$^r_{\gr} $ condition and use it in Lift$^r_{\gr}$ condition since 
if $e=xr$ and $1-e=(1-x)s,$ then $e-x=e-xe-x+xe=(1-x)e-x(1-e)=(1-x)xr-x(1-x)s=(x-x^2)(r-s).$ Conversely, if Lift$^r_{\gr}$ holds for $x\in H$ with $e=e^2\in H,$ and $e-x=(x-x^2)r,$ then $e=x+(x-x^2)r\in xR$ and $1-e=1-x-(x-x^2)r=(1-x)(1-xr)\in (1-x)R$ showing Exch$^r_{\gr}.$

Since Exch$^r_{\gr}$ $\Leftrightarrow$ Lift$^r_{\gr},$ the same argument showing Cln $\Rightarrow$ Exch 
can be used for Cln$_{\gr}$ $\Rightarrow$ Exch$^r_{\gr}.$ As Cln$_{\gr}$ is left-right symmetric, we have that Cln$_{\gr}\Rightarrow$ Exch$_{\gr}.$ 

We explore some conditions under which certain graded matrix algebras are graded exchange. 

\begin{proposition}
Let $K$ be a trivially $\Zset$-graded field, $m, n$ positive integers, and $\gamma_1,\ldots,\gamma_n\in \Zset.$ Then, the following hold.
\begin{enumerate}[\upshape(1)]
\item $\M_n(K)(\gamma_1,\ldots, \gamma_n)$ is graded exchange. 
 
\item Let $k_i$ denote the cardinality of $\{\gamma\in\{\gamma_1,\ldots, \gamma_n\}\mid \gamma\equiv i$ modulo $m\}$ for $i=0,\ldots, m-1.$ If  $k_i$ is either 0 or 1  for each $i=0,\ldots, m-1,$ then $\M_n(K[x^m, x^{-m}])(\gamma_1,\ldots, \gamma_n)$ is graded exchange.  
\end{enumerate}
\label{matrices_graded_exchange}
\end{proposition}
\begin{proof} 
Graded rings in both parts are graded involutive
(\cite{Roozbeh_Lia_Ultramatricial} has more details), so showing that Exch$_{\gr}^r$ holds ensures that  Exch$_{\gr}^l$ also holds.  

To show (1), let $S=\M_n(K)(\gamma_1,\ldots, \gamma_n).$ By Lemma \ref{lemma_on_shifts}, we can assume that $\gamma_1=0$ and $\gamma_i\leq \gamma_{i+1}$ for all $i=1,\ldots n-1.$ Then $S_l$ consists of matrices with zeros on and above the main diagonal if $l>0,$ and, $S_l$ consists of matrices with zeros on and below the main diagonal if $l<0.$ Thus, $1_S-a$ is invertible for any $a\in S_l$ if $l\neq 0.$ So, $e=0$ is such that $e\in aS$ and $1_S=1_S-e\in (1_S-a)S=S.$ Since $S_0$ is a direct sum of matrix algebras over $K,$ $S_0$ is exchange.   

To show (2), let $S=\M_n(K[x^m, x^{-m}])(\gamma_1,\ldots, \gamma_n).$ By Lemma \ref{lemma_on_shifts}, we can assume that $\gamma_j\in\{0,1,\ldots, m-1\}$ and that $\gamma_j\leq \gamma_{j+1}$ for all $j=1,\ldots n.$ 
Assume that $k_i=0$ or $k_i=1$ for all $i=0,\ldots, m-1.$ For an arbitrary homogeneous element $a$ of $S,$ we claim that either $a$ is right invertible, $1_S-a$ is right invertible, or $a$ is in the $K$-linear span of some of the standard matrix units $e_{ii}, i=1,\ldots, n.$ By taking $e=1_S$ in the first case, $e=0$ in the second case, and $e=\sum_{a_{ii}\neq 0}e_{ii}$ in the third case, we obtain a homogeneous idempotent $e$ such that $e\in aS$ and $(1_S-e)\in (1_S-a)S.$ 
 
Let $a$ be in $S_{l'}$ and let $l'=km+l,$ for some $l=0,\ldots, m-1.$ Consider the cases $l=0$ and $l\neq 0.$

If $l=0,$ then $a$ is diagonal with the diagonal entries in $Kx^{km}.$ Thus, $a=\sum_{a_{ii}\neq 0}a_{ii}e_{ii}.$

If $l\neq 0,$ we consider two cases: (i) $k_i=1$ for all $i$ and, (ii) $k_i=0$ for some $i.$ 

If all $k_i$ are 1, then $n=m,$ $S=\M_m(K(x^m, x^{-m}))(0,1, \ldots, m-1),$ and $a$ has the following form for some  $a_1, \ldots, a_m\in K.$ {\small
\[a=\left[
\begin{array}{cccc|cccc}
0 & 0 & \ldots & 0 & a_1x^{km} & 0         & \ldots & 0\\
0 & 0 & \ldots & 0 & 0         & a_2x^{km} &\ldots & 0\\ 
\vdots & \vdots & \ddots & \vdots & \vdots & \vdots & \ddots & 0\\
0 & 0 & \ldots  & 0 & 0 & 0 & \ldots  &  a_lx^{km} \\ \hline
a_{l+1}x^{(k+1)m} & 0 & \ldots & 0 & 0 & 0 & \ldots &  0\\ 
0 & a_{l+2}x^{(k+1)m} & \ldots & 0 & 0 & 0 & \ldots &  0\\
\vdots & \vdots & \ddots & \vdots & \vdots & \vdots & \ddots & \vdots\\
0 & 0 & \ldots & a_mx^{(k+1)m} & 0 & 0 & \ldots &  0\\
\end{array}
\right]\]}

\noindent We explore the condition that $1_S-a$ is right invertible. If $(1_S-a)b=1_S$ for some $b=[b_{ij}],$ 
then  
\[b_{ij}-a_ix^{km}b_{(m-l+i)j}=\delta_{ij}\mbox{ for }i=1,\ldots, l\mbox{ and }j=1,\ldots, m\;\mbox{ and}\]
\[b_{ij}-a_ix^{(k+1)m}b_{(i-l)j}=\delta_{ij}\mbox{ for }i=l+1,\ldots, m\mbox{ and }j=1,\ldots, m\]
where $\delta_{ij}$ stands for $\delta_{ij}=1$ if $i=j$ and 0 otherwise. If we fix $j=1,\ldots, m$ and consider these equations as equations in $m$ unknowns $b_{ij}, i=1,\ldots, m,$ then one can  eliminate the variables $b_{ij}$ for $i\neq j$ using the relations above and reduce the system to the following single equation in $b_{jj}.$   
\[b_{jj}\left(1-\left(\prod_{i=1}^m a_i\right) x^{m((k+1)m-l)}\right)=1\]
Thus, if any of $a_1,\ldots, a_m$ is zero, this equation has a unique solution $b_{jj}=1$ which determines the values of all other $b_{ij}$ for $i\neq j$ so $1_S-a$ is right invertible. 

If none of $a_1,\ldots, a_m$ is zero, then $a$ is an invertible matrix with the inverse 
 {\small
\[a^{-1}=\left[
\begin{array}{cccc|cccc}
0 & 0 & \ldots & 0 & a_{l+1}^{-1}x^{-(k+1)m} & 0         & \ldots & 0\\
0 & 0 & \ldots & 0 & 0         & a_{l+2}^{-1}x^{-(k+1)m} &\ldots & 0\\ 
\vdots & \vdots & \ddots & \vdots & \vdots & \vdots & \ddots & 0\\
0 & 0 & \ldots  & 0 & 0 & 0 & \ldots  &  a_m^{-1}x^{-(k+1)m} \\ \hline
a_1^{-1}x^{-km} & 0 & \ldots & 0 & 0 & 0 & \ldots &  0\\ 
0 & a_2^{-1}x^{-km} & \ldots & 0 & 0 & 0 & \ldots &  0\\
\vdots & \vdots & \ddots & \vdots & \vdots & \vdots & \ddots & \vdots\\
0 & 0 & \ldots & a_l^{-1}x^{-km} & 0 & 0 & \ldots &  0\\
\end{array}
\right].\]}  

It remains to consider the case when some $k_i$ is zero in which case $n<m.$ Adding $m-1-i$ to all shifts, we can assume that $i=m-1.$ We use  induction on $m-n$ to show that $1_S-a$ is right invertible. If $m-n=1,$ then $k_0=\ldots =k_{m-2}=1$ and $S$ embeds in 
$S'=\M_{m}(0,1, \ldots,  m-1)$ by $\phi: s\mapsto 
\left[\begin{array}{cc}
s & 0\\
0 & 0
\end{array}\right]$ for $s\in S.$ Since $\phi$ is a graded map,  $\phi(a)$ is homogeneous in $S'.$ As $\phi(a)$ is not right invertible in $S',$ $1_{S'}-\phi(a)$ is right invertible by case (i). If a right inverse of $1_{S'}-\phi(a)$ has $b$ in the upper-left block, then $b$ is a right inverse of $1_S-a.$  

Assuming the induction hypothesis, let us show the claim for $m-n>1.$ In this case, there is a graded embedding $\phi$ of $S$ in 
$S'=\M_{n+1}(k_0(0),k_1(1), \ldots,  k_{m-2}(m-2), m-1)$ where $k_i(i)$ is the empty list if $k_i=0$ and $k_i(i)=i$ if $k_i=1.$ By the induction hypothesis, $1_{S'}-\phi(a)$ is right invertible. If $b$ is the upper-left block, then $b$ is a right inverse of $1_S-a.$  
\end{proof} 

Using Proposition \ref{matrices_graded_exchange}, we show that the implication Cln$_{\gr}\Rightarrow$ Exch$_{\gr}$ is strict.

\begin{example}
Let $R=\M_2(K[x^2, x^{-2}])(0,1)$ (note that $R$ can also be realized as the Leavitt path algebra of the graph  $\xymatrix{ {\bullet} \ar@/^1pc/ [r]   & {\bullet} \ar@/^1pc/ [l]}$). By Proposition \ref{matrices_graded_exchange}, $R$ is graded exchange but, by Proposition \ref{graded_clean_matrices}, $R$ is not graded clean. Thus, Exch$_{\gr}\nRightarrow$ Cln$_{\gr}$. 
\label{example_Cln_gr_sledi_Exch_gr_is_strict}
\end{example}

Exch$^r_{\gr}$ is closed under the formation of graded corners (i.e. the corner formed by a homogeneous idempotent) and graded direct limits (the proof is completely analogous to the proof of the non-graded case). Besides these favorable properties, Exch$^r_{\gr}$ is still a rather restrictive condition. For instance, Example \ref{example_gr_reg_not_gr_exch} shows that 
Reg$_{\gr}\;\nRightarrow$ Exch$^r_{\gr}.$

\subsection{Nonunital graded exchange}\label{subsection_graded_exchange}
The equivalence of the following conditions enables us to consider Exch$^r_{\gr}$ in the nonunital case also. 

\begin{proposition}
If $R$ is a graded general ring and $\gamma\in \Gamma,$ the following conditions are equivalent. 
\begin{enumerate}[\upshape(1)]
\item For any $x\in R_\gamma,$ there is $e=e^2\in R_\varepsilon$ such that $e\in -xR$ and $e\in x\ast R.$

\item For any $x\in R_\gamma,$ there is $e=e^2\in R_\varepsilon$ such that $e\in xR$ and $e\in x\circ R.$

\item For any $x\in R_\gamma,$ there is $e=e^2\in R_\varepsilon$ such that $e\in xR$ and $(-e, 1)\in (-x, 1)R^u.$

\item For any graded embedding $\phi:R\to S$ such that $S$ is a graded unitization of $R,$ and for any $x\in R_\gamma,$ there is $e=e^2\in R_\varepsilon$ such that $e\in xR$ and  $1-\phi(e)\in (1-\phi(x))S.$ 
\end{enumerate}
If $R$ is graded unital, then $R$ is graded right exchange if and only if the above conditions hold. 
\label{graded_exchange}
\end{proposition}

\begin{proof}
Assuming (1) and using it for $-x$ shows (2). Similarly, (2) implies (1).
If (2) holds, then $e=x\circ y$ for some $y\in R.$ Thus, $(-e, 1)=(-x-y+xy, 1)=(-x,1)(-y,1)\in (-x,1)R^u$ so (3) holds. Conversely, if (3) holds and $(-e,1)=(-x,1)(y,k)$ for some $y\in R$ and $k\in \Zset,$ then $k=1$ and $-e=-x+y-xy$ so that $e=x+(-y)-x(-y)=x\circ(-y)\in x\circ R.$  Similarly, (2) $\Leftrightarrow$ (4).  
\end{proof}

We continue to say that a ring is {\em graded right exchange} and that it has  Exch$^r_{\gr}$ if the equivalent conditions above hold. Exch$^l_{\gr}$ and Exch$_{\gr}$ are similarly generalized for graded general rings. 

\begin{proposition}
The property {\em Exch}$^r_{\gr}$ is closed under the formation of graded corners and graded direct limits of graded {\em general} rings. 
\label{exchange_graded_corners_and_dir_lim}
\end{proposition}
\begin{proof}
The proof of \cite[Proposition 1.3]{Ara_exchange} 
adapts to the graded case to show the claim on  graded corners. Parts (1) and (2) of Proposition \ref{graded_exchange} can be used for the claim on graded direct limits.
\end{proof}

\subsection{Graded local cancellation properties}\label{subsection_graded_local}

Recall that a graded general ring $R$ is said to have graded local units if every finite set of homogeneous elements of $R$ is contained in a graded corner. In this case, one can meaningfully consider the graded (and nonunital) versions of ring properties by considering their graded local versions as follows. If $R$ is a graded general ring and $P$ is a ring property, we say that $R$ has {\em graded local $P$} (and that local $P_{\gr}$ holds for $R$) if every finite set of homogeneous elements of $R$ is contained in a graded corner which has the property $P_{\gr}.$ 

If $R$ is a graded unital ring, then $P_{\gr}\Leftrightarrow$ local $P_{\gr}$ for any property $P.$ For nonunital graded rings, this equivalence does not have to hold. In particular, the proof of (local UR $\Rightarrow$ UR) does not generalize to the graded case since $u\in U(\ast)\cap H$ such that $x=xux+x^2$ for some $x\in H$ can be found only if $x\in R_\varepsilon.$ This enables us to create an example of a locally UR$_{\gr}$ ring such that the nonunital generalization of UR$_{\gr}$ fails.  

\begin{example}
Consider $K[x^2, x^{-2}]$ with the standard $\Zset$-grading (see section \ref{subsection_graded_rings_prerequisites}), let $R_{2n}=\M_{2n}(K[x^2, x^{-2}])(0,1,0,1,\ldots, 1)$ and $R_{2n+1}=\M_{2n+1}(K[x^2, x^{-2}])(0,1,0,1,\ldots, 0).$ The number of shifts which are zero is equal to the number of shifts which are one for $R_{2n}$ so UR$_{\gr}$ holds on $R_{2n}$ by \cite[Proposition 5.1]{Lia_cancellation}. By the same result,  UR$_{\gr}$ does not hold on $R_{2n+1}.$ Embed $R_n$ into $R_{n+1}$ 
by $a\mapsto \left[
\begin{array}{cc}
a & 0\\
0 & 0
\end{array}\right]$ and let $R=\varinjlim_n R_n.$ Every homogeneous matrix of $R_{2n+1}$ embeds in $R_{2n+2}$ which is graded isomorphic to a graded unit-regular corner of $R$. Hence, $R$ is locally UR$_{\gr}$. However, no homogeneous $u\in U(\ast)\cap H$ such that $x=xux+x^2$ can be found for any $x\in H$ when $x$ is not in the 0-component.  
\end{example}

This example also shows that having graded corners which are not graded unit-regular is not an obstruction for a graded locally unital ring to be graded locally unit-regular. This contrasts the non-graded case in which a locally UR ring cannot have non-UR corners. 

For Exch, Exch$^r_{\gr}\Leftrightarrow$ local Exch$^r_{\gr}$ holds for all graded locally unital rings. Indeed, the direction $\Rightarrow$ holds since Exch$^r_{\gr}$ transfers to graded corners and the direction $\Leftarrow$ holds by the proof analogous to the non-graded case (see Proposition \ref{local_P_implies_P}).

\subsection{\texorpdfstring{$\varepsilon$}{TEXT}-cancellation properties}\label{subsection_varepsilon_exchange}
Besides the graded versions of ring properties, one can define the graded versions of module properties as follows. If $P(A)$ is a property of a module $A,$ we let $P_{\gr}(A)$ denote the statement on a graded module $A$ obtained by replacing every instance of ``module'' by ``graded module'' and every instance of ``homomorphism'' by ``graded homomorphism'' in $P(A).$  

Since a graded homomorphism between right modules $A$ and $B$ is an element of $\HOM_R(A,B)_\varepsilon,$  the $\varepsilon$-component of $\END_R(R_R)$ has a special significance. If $P$ is a property of graded (unital or general) rings, let $P_\varepsilon$ denote the requirement of a graded general ring $R$ that $\END_R(R_R)_\varepsilon$ has $P.$

With this definition, if $\END_R(R_R)$ is graded regular, UR$_\varepsilon$ holds for $R$ if and only if $R_R$ has graded internal cancellation.  \cite[Proposition 3.4]{Lia_cancellation} shows this in the unital case using \cite[Theorem 4.1]{Goodearl_book}. This last results holds in the nonunital case and the proof adapts to the grading case also. The condition that $R$ has sr=1$_\varepsilon$ is equivalent with the requirement that the graded module $R_R$ has graded substitution (the proof of \cite[Theorem 4.4]{Lia_cancellation} generalizes to the nonunital case also). We also have that DF$_\varepsilon$ holds for $R$ if and only if $R_R$ is graded directly finite as a graded module (the proof from \cite[Section 4.3]{Lia_cancellation} in the unital case carries to the nonunital case).  

With Cln$_\varepsilon$ and Exch$_\varepsilon$ defined analogously, the implications UR$_\varepsilon$ $\Rightarrow$ sr=1$_\varepsilon$ $\Rightarrow$ DF$_\varepsilon$ and UR$_\varepsilon$ $\Rightarrow$ Cln$_\varepsilon$ $\Rightarrow$ Exch$_\varepsilon$
are direct since the cancellation $\varepsilon$-properties reduce to the non-graded case.  

If $R$ is unital, $\End_R(R_R)=\END_R(R_R)\cong_{\gr} R,$ so $P_{\gr}\;\Rightarrow\; P_\varepsilon$ for any of the cancellation properties $P.$ When $P$ is  UR, sr=1, or DF, this implication is strict by \cite{Lia_cancellation}. The first ring from Example \ref{example_not_morita} shows that Cln$_{\gr} \nRightarrow$ Cln$_\varepsilon.$ The $0$-component of a unital Leavitt path algebra $L_K(E)$ is a matricial algebra over $K$ so Exch$_\varepsilon$ holds. If $E$ has exits, $L_K(E)$ does not have Exch$_{\gr}$ by Proposition \ref{graded_exchange_necessary}. 

The $\varepsilon$-properties are much less restrictive than the $\gr$-properties and sufficient to consider when one is interested in graded versions of module properties. In \cite{Lia_cancellation}, it is shown that UR$_\varepsilon$, sr=1$_\varepsilon$, and DF$_\varepsilon$ retain many properties of their non-graded counterparts. The same holds for Cln$_\varepsilon$ and Exch$_\varepsilon.$ For example, Cln$_\varepsilon$ is closed under the formation of graded matrix rings while Cln$_{\gr}$ is not. Indeed, if $\gamma\in \Gamma,$ then $\M_1(R)(\gamma)=\End_R((\gamma^{-1})R)\cong_{\gr}(\gamma^{-1})R(\gamma).$ The $\varepsilon$-component of this last ring is $R_{\gamma^{-1}\varepsilon\gamma}=R_\varepsilon$, so $R$ has Cln$_\varepsilon$ if and only if $\M_1(R)(\gamma)$ has Cln$_\varepsilon.$ This, combined with the proof that Cln is closed under the formation of matrix rings (\cite[Lemma, page 2598]{Han_Nicholson}) shows the claim.  

Exch$_\varepsilon$ is left-right symmetric since Exch is such. Also, requiring $R_\varepsilon$ to be exchange is sufficient for the graded version of the following important property to hold. We say that homogeneous idempotents {\em lift modulo a graded right ideal} $I$ if, for a homogeneous element $x$ such that $x-x^2\in I,$ there is $e=e^2\in R_\varepsilon$ such that $e-x\in I.$

\begin{proposition}
Let $R$ be a graded unital ring. If $R_\varepsilon$ is exchange, then homogeneous idempotents lift modulo every graded right ideal.  
\label{lifting_idempotents}
\end{proposition}
\begin{proof}
Let $I$ be a graded right ideal, $x\in R_\gamma$ and $x-x^2\in I.$ Consider the cases $\gamma=\varepsilon$ and $\gamma\neq \varepsilon.$ In the first case, $x-x^2$ is in $I_\varepsilon$ which is a right ideal of $R_\varepsilon.$ Since $R_\varepsilon$ is exchange, Lift holds so there is $e=e^2\in R_\varepsilon$ such that $e-x\in (x-x^2)R_\varepsilon\subseteq I_\varepsilon\subseteq I.$ If $\gamma\neq\varepsilon,$ then $x\in I_\gamma\subseteq I$ (and $x^2\in I_{\gamma^2}$) since $I$ is graded. So, one can take $e=0$ and have $e-x=-x\in I.$ 
\end{proof}

If $R$ is graded right exchange, then $R_\varepsilon$ is exchange, so the conclusion of Proposition \ref{lifting_idempotents} holds. This conclusion is not sufficient for $R$ to be graded exchange since the right ideal $(x-x^2)R$ is not graded if $x\in R_\gamma$ and $\gamma\neq\varepsilon$. In particular, Example \ref{example_gr_reg_not_gr_exch} exhibits a $\Zset$-graded ring $R$ such that $R_0$ is exchange (so the conclusion of Proposition \ref{lifting_idempotents} holds) and $R$ is not graded exchange.

If $R$ is a graded unital ring and $A$ is a graded right $R$-module, let FinExch$_{\gr}$($A$) stand for the graded version of FinExch($A$). The proof of \cite[Theorem 2.1]{Nicholson_clean} 
directly translates to the proof of the equivalence 
\[\mbox{FinExch}_{\gr}(A_R)\mbox { holds }\;\Leftrightarrow\; \mbox{ Exch holds for }\END_R(A)_\varepsilon.\label{*}\tag{*}
\]
Thus, FinExch$_{\gr}(R_R)$ holds if and only if $R$ has Exch$_\varepsilon$. If Morita equivalence is suitably defined for graded unital rings (as in \cite[Section 2.3]{Roozbeh_book}), the following holds.  

\begin{proposition}
The condition {\em Exch}$_\varepsilon$ is graded Morita invariant for graded unital rings. 
\label{epsilon_exchange}
\end{proposition}
\begin{proof}
If $A$ and $B$ are graded right $R$-modules, FinExch$_{\gr}(A\oplus B)$ holds if and only if FinExch$_{\gr}(A)$ and FinExch$_{\gr}(B)$ hold (the proof of \cite[Lemma 3.10]{Crawley_Jonsson} adapts to the graded case). We also claim that FinExch$_{\gr}(A)$ holds if and only if FinExch$_{\gr}((\gamma)A)$ holds for any graded right module $A$ and any $\gamma\in \Gamma.$ Indeed, if $(\gamma)A$ is a direct summand of some module $B=B_1\oplus B_2,$ and  FinExch$_{\gr}(A)$ holds, then $A$ is a direct summand of $(\gamma^{-1})B=(\gamma^{-1})B_1\oplus (\gamma^{-1})B_2.$ Thus, there are  graded submodules $A_i$ of $(\gamma^{-1})B_i$ such that $A_1\oplus A_2$ is a complement of $A$ in $(\gamma^{-1})B.$ The modules $(\gamma)A_i$ are  graded submodules of $B_i$ such that $(\gamma)A_1\oplus (\gamma)A_2$ is a complement of $(\gamma)A$ in $B.$ Because of the equivalence (\ref{*}), this shows that the property Exch$_\varepsilon$ is closed under the formation of graded matrix algebras. 

If $e$ is a homogeneous idempotent in $R$, then $(eRe)_\varepsilon=eR_\varepsilon e.$ Thus, the fact that Exch is closed under the formation of corners directly implies that Exch$_\varepsilon$ is closed under the formation of graded corners. Hence, Exch$_\varepsilon$ is graded Morita invariant.
\end{proof}

\section{Graded cancellation properties of Leavitt path algebras}\label{section_LPAs}

\subsection{Leavitt path algebras review}
\label{subsection_LPAs} 

Let $E$ be a directed graph, let $E^0$ denote the set of vertices, $E^1$ the set of edges, and  $\so$ and $\ra$ denote the source and the range maps of $E.$  If both $E^0$ and $E^1$ are finite, $E$ is a {\em finite} graph. A {\em sink} of $E$ is a vertex which does not emit edges. A vertex of $E$ is {\em regular} if it is not a sink and if it emits finitely many edges and $E$ is said to be {\em row-finite} if every vertex is either a sink or it is regular. 
A {\em path} is a sequence of edges $e_1e_2\ldots e_n$ such that  $\ra(e_i)=\so(e_{i+1})$ for $i=1,\ldots, n-1$ or a single vertex. A {\em cycle} is a closed path such that different edges in the path have different sources. The graph is {\em acyclic} if there are no cycles. A cycle has {\em an exit} if a vertex on the cycle emits an edge outside of the cycle. The graph is {\em no-exit} if no cycle has an exit and it has {\em Condition (K)} if no vertex is on exactly one cycle. 

Extend a graph $E$ to the graph with the same vertices and with edges $E^1\cup \{e^*\mid e\in E^1\}$ where the range and the source maps on the added edges are given by $\so(e^*)=\ra(e)$  and $\ra(e^*)=\so(e).$ If $K$ is any field, the \emph{Leavitt path algebra} $L_K(E)$ of $E$ over $K$ is a free $K$-algebra generated by the set  $E^0\cup E^1\cup\{e^\ast\mid e\in E^1\}$ such that for all vertices $v,w$ and edges $e,f,$

\begin{tabular}{ll}
(V)  $vw =0$ if $v\neq w$ and $vv=v,$ & (E1)  $\so(e)e=e\ra(e)=e,$\\
(E2) $\ra(e)e^\ast=e^\ast\so(e)=e^\ast,$ & (CK1) $e^\ast f=0$ if $e\neq f$ and $e^\ast e=\ra(e),$\\
(CK2) $v=\sum_{e\in \so^{-1}(v)} ee^\ast$ for each regular vertex $v.$ &\\
\end{tabular}

By the first four axioms, every element of $L_K(E)$ can be represented as a sum of the form $\sum_{i=1}^n k_ip_iq_i^\ast$ for some $n$, paths $p_i$ and $q_i$, and elements $k_i\in K,$ for $i=1,\ldots,n$ where $v^*=v$ for $v\in E^0$ and $p^*=e_n^*\ldots e_1^*$ for a path $p=e_1\ldots e_n.$ Using this representation, one can make $L_K(E)$ into an involutive ring by $\left(\sum_{i=1}^n k_ip_iq_i^\ast\right)^*=\sum_{i=1}^n k_i^*q_ip_i^\ast$ where $k_i\mapsto k_i^*$ is any involution on $K$. It is direct to see that $L_K(E)$ is locally unital (with the finite sums of vertices as the local units) and that $L_K(E)$ is unital if and only if $E^0$ is finite in which case the sum of all vertices is the identity. For more details on these basic properties, see \cite{LPA_book}.

We note a short lemma we need later on. 
\begin{lemma}
If $p=e_1\ldots e_n$ is a path, then $pp^*=\so(p)$ if and only if $\so(e_i)$ emits only $e_i$ for every $i=1,\ldots, n.$ 
\label{positive_definite}
\end{lemma}
\begin{proof}
We use induction on $n.$ If $n=1,$ $p=e\in E^1,$ and $ee^*=\so(e),$ assume that $\so(e)$ emits more than one edge so that  $S=\{f\in \so^{-1}(\so(e))\mid f\neq e\}\neq \emptyset.$ 
If $\leq$ stands for the order on the set of projections (selfadjoint idempotents) given by $p\leq q$ iff $pq=qp=p,$ then $ee^*\leq ee^*+ff^*$ and $ee^*=\so(e)\geq ee^*+ff^*$ for any $f\in S.$ Thus, $ee^*+ff^*=ee^*$ which implies that $ff^*=0$ and so $f=f\ra(f)=ff^*f=0f=0$. This is a contradiction since $f$ is a basis element of $L_K(E).$ 

Assuming the induction hypothesis, let $p=eq$ for $e\in E^1$ and a path $q$ with $|q|\geq 1.$ If $pp^*=\so(p),$ multiply this relation by $e^*$ on the left and by $e$ on the right to have that $qq^*=\ra(e)=\so(q).$ The induction hypothesis applies to $q,$ so it remains to show that $\so(e)$ emits no other edges but $e.$ This holds also by the induction hypothesis since $ee^*=e\so(q)e^*=eqq^*e^*=pp^*=\so(p)=\so(e).$

The converse is direct by the (CK2) and (E1) axioms. 
\end{proof}

If we consider $K$ to be trivially graded by $\Zset,$ $L_K(E)$ is naturally graded by $\Zset$ so that the $n$-component $L_K(E)_n$ is the $K$-linear span of the elements $pq^\ast$ for paths $p, q$ with $|p|-|q|=n$ where $|p|$ denotes the length of a path $p.$ While one can grade a Leavitt path algebra by any group $\Gamma$ (see \cite[Section 1.6.1]{Roozbeh_book}), we always consider the natural grading by $\Zset.$ 

\subsection{Cancellation and graded cancellation properties of Leavitt path algebras}

By \cite[Theorems 1 and 2]{Gene_Ranga_regular}, the conditions that $L_K(E)$ is locally unit-regular, that $L_K(E)$ is regular, and that $E$ is acyclic are equivalent. 
By Proposition \ref{local_P_implies_P}, local unit-regularity implies unit-regularity so if $E$ is acyclic, $L_K(E)$ is unit-regular. The converse holds since unit-regularity implies regularity which, for a Leavitt path algebra, implies that the underlying graph is acyclic. 

By \cite[Theorem 4.12]{Lia_traces}, $L_K(E)$ is locally directly finite if and only if $E$ is no-exit. This is equivalent with $L_K(E)$ being directly finite since DF $\Rightarrow$ local DF as DF transfers to corners and since local DF $\Rightarrow$ DF by Proposition \ref{local_P_implies_P}.

We claim that $L_K(E)$ has sr=1 if and only if $E$ is acyclic. If $E$ is acyclic, then $L_K(E)$ is a direct limit of matricial algebras over $K.$ Such direct limit has local sr=1 and so $L_K(E)$ has sr=1 by Proposition \ref{local_P_implies_P}. Conversely, if $L_K(E)$ has sr=1, then it has DF by Proposition \ref{DF_and_corners}, so $E$ is no-exit. Also, since sr=1 is closed under formation of corners by \cite[Theorem 3.9]{Vaserstein2}, every corner of $L_K(E)$ has sr=1. Assuming that there is a cycle in $E$, let $v$ be any vertex of this cycle. By \cite[Lemma 2.2.7]{LPA_book}, $vL_K(E)v$ is isomorphic to $K[x,x^{-1}]$ and this algebra does not have sr=1. Thus, $E$ is acyclic.

By \cite[Theorem 4.5]{Gonzalo_Pardo_Molina_exchange}, $L_K(E)$ is exchange if and only if $E$ has Condition (K). Thus, if $L_K(E)$ is clean, $E$ has Condition (K). If $E$ is acyclic, $L_K(E)$ has Reg + sr=1, so Cln as well. It is not known exactly which class of graphs sandwiched between the acyclic graphs and graphs with Condition (K) corresponds to clean Leavitt path algebras. The diagram below summarizes the relations. 

{\small
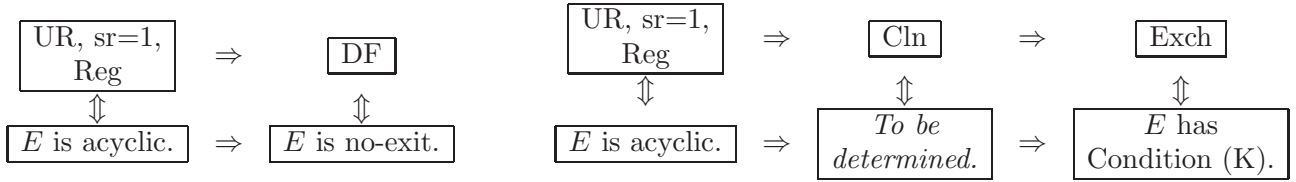
\begin{center}
\begin{tabular}{ccc}
\begin{tabular}{|c|} \hline 
UR, sr=1,\\ Reg\\ \hline
\end{tabular}
& $\Rightarrow$ & 
\begin{tabular}{|c|} \hline 
DF  \\ \hline
\end{tabular}\\
$\Updownarrow$&&$\Updownarrow$\\
\begin{tabular}{|c|} \hline 
$E$ is acyclic.\\ \hline
\end{tabular}
& $\Rightarrow$ & 
\begin{tabular}{|c|} \hline 
$E$ is no-exit. \\ \hline
\end{tabular}\\
\end{tabular}\hskip1cm
\begin{tabular}{ccccc}
\begin{tabular}{|c|} \hline 
UR, sr=1,\\ Reg\\ \hline
\end{tabular}
& $\Rightarrow$ & 
\begin{tabular}{|c|} \hline 
Cln  \\ \hline
\end{tabular}
& $\Rightarrow$& 
\begin{tabular}{|c|} \hline 
Exch\\ \hline
\end{tabular}\\
$\Updownarrow$&&$\Updownarrow$&&$\Updownarrow$\\
\begin{tabular}{|c|} \hline 
$E$ is acyclic.\\  \hline
\end{tabular}
& $\Rightarrow$ & 
\begin{tabular}{|c|} \hline 
{\em To be}\\{\em determined.}\\ \hline
\end{tabular}
& $\Rightarrow$& 
\begin{tabular}{|c|} \hline 
$E$ has\\ Condition (K).\\ \hline
\end{tabular}\\
\end{tabular}
\end{center}}

We turn to the graded cancellation properties of Leavitt path algebras now. The 0-component $L_K(E)_0$ is an ultramatricial algebra over $K$ so it is locally unit-regular. Thus, the conditions UR$_\varepsilon,$ sr=1$_\varepsilon,$ DF$_\varepsilon,$ Cln$_\varepsilon,$ and Exch$_\varepsilon$ hold for $L_K(E)$ of any graph $E.$ 

By \cite[Theorem 3.7]{Roozbeh_Ranga_Ashish}, $L_K(E)$ is graded locally directly finite if and only if $E$ is no-exit. If $E^0$ is finite, these conditions are equivalent with $L_K(E)$ being graded directly finite. We consider the graded versions of other cancellation properties for Leavitt path algebras next.

\subsection{Graded clean Leavitt path algebras}
For Leavitt path algebras, the condition Cln$_{\gr}$ is the most restrictive of all $\gr$-cancellation properties.  

\begin{proposition}
If $K$ is a field and $E$ is a graph with finitely many vertices, the following conditions are equivalent. 
\begin{enumerate}[\upshape(1)]
\item $L_K(E)$ is graded clean. 
\item $E$ is either a collection of disjoint vertices with no edges between them or $E$ is $\xymatrix{\bullet\ar@(ru,rd) }\;\;\;\;\,$.
\end{enumerate} 
\label{LPA_graded_clean_characterization} 
\end{proposition}
\begin{proof}
If $L_K(E)$ is graded clean and $E$ has edges, they are invertible by Lemma \ref{graded_clean_lemma}. Assuming that $E$ has more than one vertex and some edges, let $u$ be the inverse of  an edge $e$ and let $v$ be a vertex different than $\so(e).$ Then $ve=v\so(e)e=0$ and so $0=veu=v\neq 0$ which is a contradiction. Thus, $\so(e)=1$ is the only vertex in $E$. Since $e^*e=\so(e)=1=ue,$ $e^*=u.$ This implies that   $ee^*=1=\so(e)$ and Lemma \ref{positive_definite} ensures that $e$ is the only edge $\so(e)$ emits.

Conversely, if $E$ is a collection of disjoint vertices, then $L_K(E)$ is graded isomorphic to a sum of finitely many copies of $K,$ graded trivially, and this ring is graded clean. If $E$ is  $\xymatrix{\bullet\ar@(ru,rd) }\;\;\;\;\,$, then $L_K(E)$ is graded isomorphic to $K[x,x^{-1}]$ which is a graded field and, hence, graded clean. 
\end{proof}

\subsection{Graded unit-regular Leavitt path algebras} 
We move on to the condition UR$_{\gr}.$ Let (EDL) be the graph property below.  
\begin{itemize}
\item[(EDL)] For every cycle of length $m,$ the lengths, considered modulo $m$, of all paths which do not contain the cycle and which end in an arbitrary but fixed vertex of the cycle, are
\[0,0,\ldots, 0,\; 1, 1, \ldots, 1,\; \ldots\ldots \ldots,\; m-1, m-1\ldots m-1 \]
where each $i$ is repeated the same number of times in the above list for $i=0, \ldots, m-1.$
\end{itemize}
The notation EDL shortens ``equally distributed lengths''. 
By \cite[Theorem 5.3]{Lia_cancellation}, if $E$ is a finite graph, $L_K(E)$ has UR$_{\gr}$ if and only if $E$ is a no-exit graph such that every sink is isolated and that Condition (EDL) holds. We relax the assumption for this result by requiring only $E^0$ to be finite. 

\begin{lemma}
If $L_K(E)$ is unital and graded unit-regular, then $E$ is row-finite.  
\label{lemma_for_UR_gr}
\end{lemma}
\begin{proof}
Assume that $v\in E^0$ emits infinitely many edges. If $e\in \so^{-1}(v),$ let $u$ be an invertible element in $R_{-1}$ such that $e=eue.$ Since $u^{-1}\in R_1,$ $u^{-1}$ can be written as a $K$-linear combination of monomials $pq^*$ where path $p$ is one edge longer than path $q.$ Let $f\in \so^{-1}(v)$ be an edge different than any edge appearing in any paths $p,q$ of such representation of $u^{-1}$ so that $f^*u^{-1}=0$ by (CK1). Multiplying the relation $u^{-1}u=1$ by $f^*$ on the left we obtain that $0=f^*$ which is a contradiction. 
\end{proof}

\begin{proposition}
If $L_K(E)$ is unital, the following conditions are equivalent. 
\begin{enumerate}[\upshape(1)]
\item $L_K(E)$ is graded unit-regular. 
 
\item $L_K(E)$ has graded stable range one.  

\item $E$ is finite, no-exit graph such that every sink is isolated and Condition (EDL) holds. 
\end{enumerate}
\label{graded_UR_characterization}
\end{proposition}
\begin{proof}
Since every Leavitt path algebra is graded regular (by \cite[Theorem 9]{Roozbeh_regular}), the first two conditions are equivalent (see \cite[Proposition 4.2]{Lia_cancellation}). 
If (1) holds, then $E$ is a finite graph by Lemma \ref{lemma_for_UR_gr}. Thus, \cite[Theorem 5.3]{Lia_cancellation} ensures that (1) $\Leftrightarrow$ (3).  
\end{proof}

Although UR$_{\gr}$ and Cln$_{\gr}$ are independent in general, Cln$_{\gr}\Rightarrow$ UR$_{\gr}$ for  unital Leavitt path algebras by Propositions \ref{LPA_graded_clean_characterization} and \ref{graded_UR_characterization}. 

\subsection{Graded exchange Leavitt path algebras} Note that a Leavitt path is a graded involutive ring, so the conditions Exch$_{\gr}^r$ and Exch$_{\gr}^l$ are equivalent. In the next result, we formulate a necessary condition for a Leavitt path algebra to be graded exchange. 

\begin{proposition}
For an arbitrary graph $E$, if $L_K(E)$ is graded exchange, then $E$ is no-exit.   
\label{graded_exchange_necessary} 
\end{proposition} 
\begin{proof}
Let $R=L_K(E)$ be graded exchange. Assume that there is a cycle $c$ with an edge exiting $c.$ If $v$ is the source of that exit, consider $c$ so that $v=\so(c)=\ra(c).$

Let $|c|=m>0.$ Since $c\in R_m$, there is $u=u^2\in R_0$ such that $u\in cR$ and $u\in c\circ R.$ The first relation implies that $u=cr$ for some $r\in R.$ Since $c=cv,$ we can replace $r$ by $vr$ so that $r=vr.$ 

As $u\in c\circ R,$  $u=c+s-cs$ for some $s\in R.$
Let $s_n$ denote the $n$-component of $s.$ If there is $k>0$ such that $s_{-k}\neq 0,$ then we can take $k$ to be the largest such so that $s_{-l}=0$ for all $l>k.$ By considering the $(-k)$-component of the relation $u=c+s-cs,$ we have that $0=s_{-k}-cs_{-m-k}=s_{-k}$ which is a contradiction with the choice of $k$. Hence, $s_{-k}=0$ for all $k>0.$ 

Considering the 0-component of $u=c+s-cs,$ we have that $u=s_0.$ Considering the $i$-component of the same relation for $i=1,\ldots, m-1$ implies that $s_i=0$ and this ensures that $s_i=0$ for each $i$ which is not a multiple of $m$. 

Considering the $m$-component of $u=c+s-cs,$ we have that $0=c+s_m-cs_0 \Rightarrow s_m=cs_0-c=cu-c.$ Considering the $km$-component for $k=2, 3,\ldots,$ we obtain that $s_{km}=c^{k-1}(cu-c).$ Since only finitely many components of $s$ are nonzero, there is a positive integer $k$ such that $c^{k-1}(cu-c)=0.$ Multiplying this relation by $(c^*)^{k-1}$, we obtain that $cu-c=0$ and so $c=cu.$ 

Combining the relations $u=cr$ and $c=cu,$ we have that $c=c^2r.$ Multiplying by $c^*$ on the left produces $c^*c=c^*c^2r \Rightarrow$ $v=cr=u.$ On the other hand, multiplying $c=c^2r$ with $(c^*)^2$ on the left produces $c^*=vr=r$ so that $u=cr=cc^*.$ Thus, $cc^*=u=v=c^*c.$ By Lemma \ref{positive_definite}, $c$ has no exits which is a contradiction with the choice of $c.$
\end{proof}

Using Proposition \ref{graded_exchange_necessary}, we can obtain examples showing Reg$_{\gr}$ $\nRightarrow$ Exch$_{\gr}^r$ and Exch $\nRightarrow$ Exch$_{\gr}^r.$
\begin{example}
Let $E$ be the graph $\;\;\;\;\xymatrix{\ar@(lu,ld)\bullet\ar@(ru,rd) }\;\;\;\;.$ Since $E$ has a cycle with an exit, $L_K(E)$ is not graded exchange by Proposition \ref{graded_exchange_necessary}. On the other hand, $E$ does have Condition (K) so $L_K(E)$ is exchange by \cite[Theorem 4.5]{Gonzalo_Pardo_Molina_exchange} and $L_K(E)$ is graded regular by \cite[Theorem 9]{Roozbeh_regular} (stating that every Leavitt path algebra is graded regular). 
\label{example_gr_reg_not_gr_exch}
\end{example}

Together with the previous results, this shows that the implications and the equivalences of the diagram (\ref{D2}) in the introduction hold and that each horizontal implication in that diagram is strict. 

If $E$ is a no-exit graph, $L_K(E)$ is graded isomorphic to a direct limit of algebras which are a finite direct sum of $\M_n(K)(\gamma_1, \ldots, \gamma_n)$ and $\M_n(K[x^m, x^{-m}])(\gamma_1,\ldots, \gamma_n)$ for various $n$ and $m$ by \cite[Proposition 3.3]{Lia_no-exit}. The algebras of the first type are graded exchange by Proposition \ref{matrices_graded_exchange}. Using the definition of $k_i$ from Proposition \ref{matrices_graded_exchange}, the algebras of the second type are graded exchange if $k_i<2,$ but at this point, we make no claim on what happens if $k_i\geq 2.$ We also note that $k_i>0.$ Indeed, recall that $n$ corresponds to the number of paths which end in an arbitrarily selected vertex of a cycle of length $m$ but without considering the cycle itself. The shifts $\gamma_1, \ldots, \gamma_n$ correspond to the lengths of these paths. There is at least one set of paths of lengths $0,1,\ldots m-1$  which correspond to the paths of the cycle ending at the selected vertex. Thus, $k_i>0.$ 

The above arguments also show that the converse of Proposition \ref{graded_exchange_necessary} would hold if the question whether $\M_n(K[x^m, x^{-m}])(\gamma_1, \ldots, \gamma_n)$ is graded exchange for every $k_i>0$ had an affirmative answer. With the question still unanswered currently,
we only have one direction both in Proposition \ref{graded_exchange_necessary} and in the next proposition, containing a sufficient condition for $L_K(E)$ to be graded exchange. 

\begin{proposition}
If $E$ is a disjoint union of acyclic graphs and graphs consisting of a single cycle each, then $L_K(E)$ is graded exchange. 
\label{graded_exchange_sufficient} 
\end{proposition}
\begin{proof}
If $E$ is as specified, then $L_K(E)$ is graded isomorphic to a direct sum of algebras of two types: first, direct limits of finite sums of algebras of the form $\M_n(K)(\gamma_1, \ldots, \gamma_n)$ and, second, algebras $\M_m(K[x^m, x^{-m}])(0, 1,\ldots, m-1)$ for various positive integers $m.$ By Propositions \ref{exchange_graded_corners_and_dir_lim} and \ref{matrices_graded_exchange}, $L_K(E)$ is graded exchange.
\end{proof}

By Propositions \ref{graded_exchange_necessary} and \ref{graded_exchange_sufficient}, the class of graphs characterizing graded exchange Leavitt path algebras strictly contains the acyclic graphs and it is contained in the class of no-exit graphs.

\end{document}